\documentclass[runningheads, envcountsame]{llncs}

\usepackage[dvipsnames]{xcolor}
\usepackage[hidelinks]{hyperref}

\usepackage{cite}
\usepackage{microtype}

\usepackage{amsmath,amssymb,amsthm,thmtools} 
\usepackage[nameinlink,capitalise]{cleveref}

\usepackage{graphicx}
\usepackage{xspace}

\usepackage{enumerate}

\DeclareMathOperator{\qn}{qn}

\def\width{\textsf{width}}
 
\input{pdffig.sty}
\graphicspath{{./}{./Figures/}}

\begin{document}

\title{On the Queue-Number of Partial Orders}

\author{
 Stefan Felsner\inst{1}
 \and
 Torsten Ueckerdt\inst{2}
 \and
 Kaja Wille\inst{1}
}%
\authorrunning{S. Felsner et al.}
\institute{
Institut f\"ur Mathematik, Technische Universit\"at Berlin, Germany
\email{felsner@math.tu-berlin.de, wille@campus.tu-berlin.de}
\and
Institute of Theoretical Informatics, Karlsruhe Institute of Technology, Germany
\email{torsten.ueckerdt@kit.edu}}

\maketitle

\begin{abstract}
  The queue-number of a poset is the queue-number of its cover graph viewed as
  a directed acyclic graph, i.e., when the vertex order must be a linear extension of
  the poset.  Heath and Pemmaraju conjectured that every poset of width $w$
  has queue-number at most $w$. Recently, Alam et al.~constructed posets of
  width $w$ with queue-number $w+1$. Our contribution is a construction of
  posets with width $w$ with queue-number $\Omega(w^2)$.
  This asymptotically matches the known upper bound.
  
\end{abstract}

\section{Introduction}
A \emph{queue layout} of a graph consists of a total ordering on its vertices and a
partition of its edge set into queues, i.e., no two edges in a single block of the
partition are nested. The minimum number of queues needed in a queue layout of a graph $G$ is its queue-number and denoted by~$\qn(G)$.

To be more precise, let $G$ be a graph and let $L$ be a linear order of the vertices.
A \emph{$k$-rainbow} is a set of~$k$ edges $\{ a_ib_i : 1\leq i\leq k \}$ such that
$a_1 < a_2 < \cdots < a_k < b_k < \cdots < b_2 < b_1$ in~$L$. A pair of edges
forming a 2-rainbow is said to be \emph{nested}. A \emph{queue} is a set
of edges without nesting.
Given $G$ and $L$, the edges of $G$ can be
partitioned into $k$ queues if and only if there is no rainbow of size $k+1$ in $L$.
The \emph{queue-number} of $G$ is the minimum number of queues needed to partition the edges of $G$ over all linear orders $L$.

Queue layouts were introduced by Heath and Rosenberg in 1992~\cite{HR-92} as
a counterpart of book embeddings. Queue layouts were implicitly used before and
have applications in fault-tolerant processing, sorting with parallel queues,
matrix computations, scheduling parallel processes, and in communication management in
distributed algorithm (see \cite{HLR-92,HR-92,NOdMW12}). There is a rich
literature exploring bounds on the queue-number of different
classes of graphs \cite{HLR-92,HR-92,Wie-17,DJMMUW-20}.

Here we study the queue-number of posets.
This parameter was introduced in 1997 by Heath and Pemmaraju~\cite{HP-97}, inspired by the older concept of the queue-number of directed acyclic graphs.
For a queue layout of a directed acyclic graph, it is required that $a$ precedes $b$ in the total vertex ordering whenever there is a directed edge
$a \to b$.
I.e., it is a topological ordering of the graph.

A \emph{poset} is a pair $P = (X,<)$ of a finite set $X$ of elements, called the ground set, and a transitive (if $a < b$ and $b < c$, then $a < c$) and antisymmetric (if $a < b$, then $b \not< a$) binary relation $<$ on $X$.
Two elements $a,b$ are called comparable if either $a < b$ or $b < a$, and incomparable otherwise.
A relation $a < b$ in $P$ is a \emph{cover} if it is not implied by transitivity, i.e., there is no element $c$ such that $a < c < b$.
In the context of drawings, embeddings and layouts for posets $P = (X,<)$, it is natural to work with their \emph{directed cover graphs}, having vertex set $X$ and a directed edge $a \to b$ for every cover relation $a < b$ in $P$. 
For example a \emph{diagram} of $P$ is an upward drawing of the directed cover graph where the direction on edges is usually omitted as each edge is implicitly directed upwards.

Now, a \emph{linear extension} $L$ of $P$ is simply a topological ordering of its directed cover graph, and we write $a < b$ in $L$ if $a$ precedes $b$ in $L$ (though not necessarily in $P$).
The \emph{queue-number} of $P$, denoted by $\qn(P)$, is the smallest $k$
such that there is a linear extension $L$ of $P$ for which the resulting linear
layout of the directed cover graph contains no $(k+1)$-rainbow. \cref{fig:small-ex} shows an example.

\calc_figscale{20}%
\begin{figure}[htb]
    \centerline{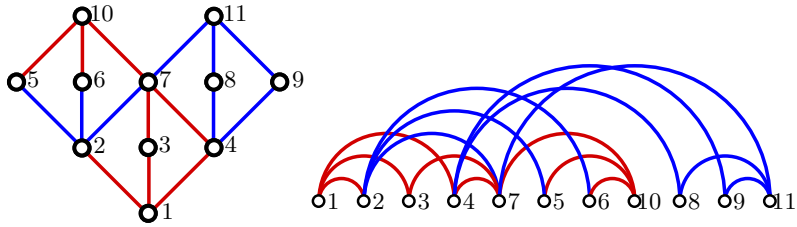}
    \caption{A poset of width 5 and a queue layout
  with 2 queues indicated by colors.\label{fig:small-ex}}
    \end{figure}%


Clearly, if $G_P$ denotes the undirected cover graph of $P$, then $\qn(G_P) \leq \qn(P)$, i.e., the queue-number of a poset is at least
as large as the queue-number of its (undirected) cover graph. It was shown by
Heath and Pemmaraju~\cite{HP-97} that even for planar posets $P$ there is no
function $f$ such that $\qn(P) \leq f(\qn(G_P))$.  They also investigated the
maximum queue-number of several classes of posets, in particular with respect
to bounded width (the maximum number of pairwise incomparable elements) and
height (the maximum number of pairwise comparable elements). In particular
they gave a nice argument showing that $\qn(P) \leq \width(P)^2$ (see
Proposition~\ref{prop:HP} below). The poset $P$ of height 2 and width $w$
whose cover graph is the complete bipartite graph $K_{w,w}$ attains
$\qn(P) = \width(P)$.  Actually, Heath and Pemmaraju conjectured that
$\qn(P) \leq \width(P)$ for every poset $P$.

Knauer, Micek and the second author~\cite{KMU-18} showed that the inequality
$\qn(P) \leq \width(P)$ holds for all posets of width 2.  Last year, Alam et
al.~\cite{ABGKP-20} constructed a non-planar poset $P_3$ of width $3$ whose queue-number is $4$; thus refuting the conjecture of Heath and Pemmaraju.
Using a simple lifting argument from~\cite{KMU-18}, Alam et al.~generalized their example and constructed for every $w > 3$ a
poset~$P_w$ with $\width(P_w) = w$ and $\qn(P_w)=w+1$.
\cref{fig:Alam-construction} shows their construction.
In fact, consider the lifting construction in the middle of \cref{fig:Alam-construction} and a fixed linear extension $L$.
If $a<b$ in $L$, then the cover edge from the bottommost element to $b$ nests above the lower copy of $P_{w-1}$.
Symmetrically, if $b<a$ in $L$, the cover edge from $b$ to the topmost element nests above the upper copy of $P_{w-1}$.
In any case, we extend any rainbow in $P_{w-1}$ by one edge.
Similarly, in the right of \cref{fig:Alam-construction} one of the diagonal cover edges will nest above one of the copies of $P_{w-1}$ in any linear extension.

Let us also mention that a second contribution of
Alam et al.~consists in a slight improvement of the upper bound: They show
$\qn(P) \leq (w-1)^2+1$ for all posets~$P$ of width at most $w$.

\begin{figure}[t]
 \centering
 \includegraphics{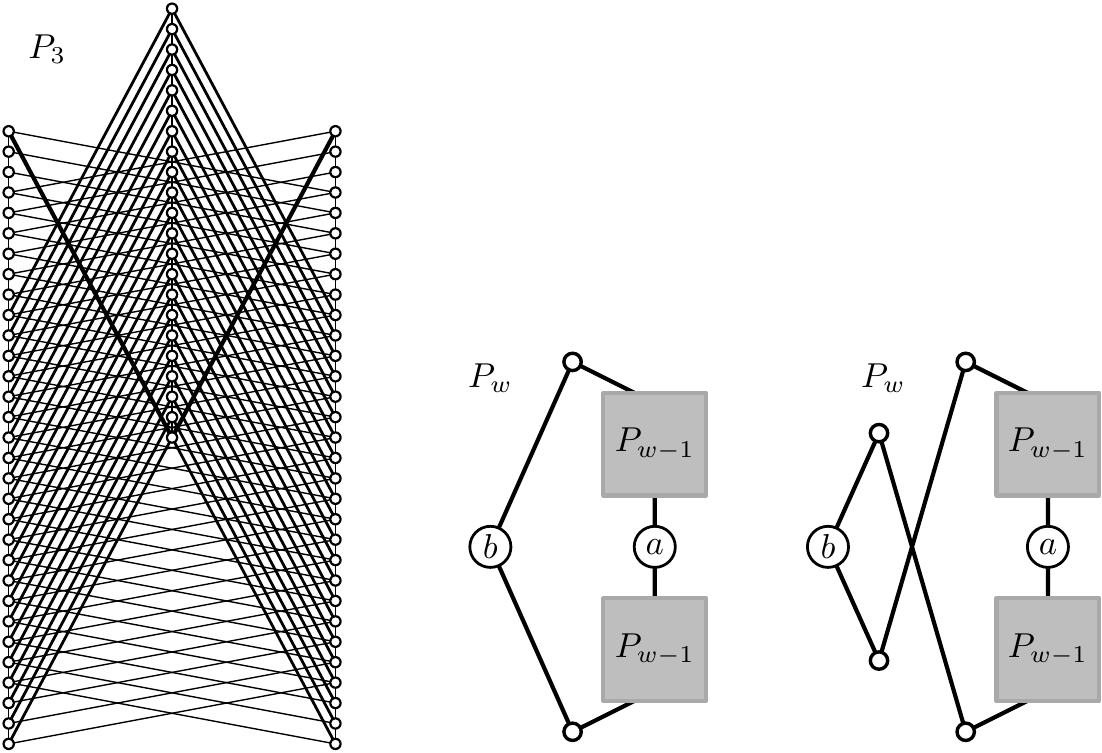}
 \caption{Left: The construction of Alam et al.~of a poset $P_3$ of width $3$ and queue-number $4$.
 Middle and right: Two possibilities of lifting a poset $P_{w-1}$ s.t. $\width(P_w)=\width(P_{w-1})+1$ and $\qn(P_w) \geq \qn(P_{w-1}) + 1$.}
 \label{fig:Alam-construction}
\end{figure}

Our contribution is the following theorem.

\begin{theorem}\label{thm:main}
  For every $w > 3$ there is a poset $P_w$ of width $w$ with
  \[
   \qn(P_w) \geq w^2/8.
  \]
\end{theorem}

These examples (asymptotically) match the upper bound. Besides yielding a strong
improvement of the lower bound, we also believe that
the analysis of our construction is conceptually simpler than the example provided by Alam et
al.~to disprove the conjecture of Heath and Pemmaraju.
The key difference is that we improve the lifting step rather than the base case.
In particular, we show how to lift any poset of width $w$ so that the width goes up by only $2$, but the queue-number goes up by at least $\lceil (w-1)/2 \rceil$.

As an open problem we promote the question whether the original conjecture
holds for planar posets. In~\cite{KMU-18} it was shown that the queue-number
of planar posets of width $w$ is upper bounded by $3w-2$ and that there are such
planar posets $P$ with $\qn(P) = \width(P) = w$.

\section{Preliminaries}

Before presenting our construction, we like to revisit the nice
upper bound argument of Heath and Pemmaraju. 
Let $P = (X, <)$ be a poset of width $w$. Dilworth's Theorem asserts that $X$
can be decomposed into $w$ chains of $P$.

\begin{proposition}[Heath and Pemmaraju]\hfill\hbox{}\par\smallskip
  \label{prop:HP}
  For every poset $P$ we have $\qn(P) \leq \width(P)^2$.
\end{proposition}

\begin{proof}
  Let $w= \width(P)$, let $C_1, \ldots , C_w$ be a chain partition, and let $L$ be any linear extension of $P$.
  Partition the cover edges into $w^2$ sets $Q_{i,j}$ with $i,j \in [w]$ such that
  $(u,v) \in Q_{i,j}$ if $u\in C_i$ and $v\in C_j$. We claim that each $Q_{i,j}$
  is a queue.
 
  Let $a < b < c < d$ in $L$ support a pair of nesting cover edges and
  suppose that both edges $(a,d)$ and $(b,c)$ belong to $Q_{i,j}$.
  By definition $a,b \in C_i$ and $c,d \in C_j$ and from the ordering in
  $L$ we get $a < b$ and $c< d$ in $P$. Now we have $a < b$ and $b<c$ and $c< d$
  in $P$ and thus the relation $a<d$ is implied by transitivity. This contradicts
  that $(a,d)$ is a cover edge.  
\end{proof}

\begin{figure}[t]
 \centering
 \includegraphics[scale=0.8]{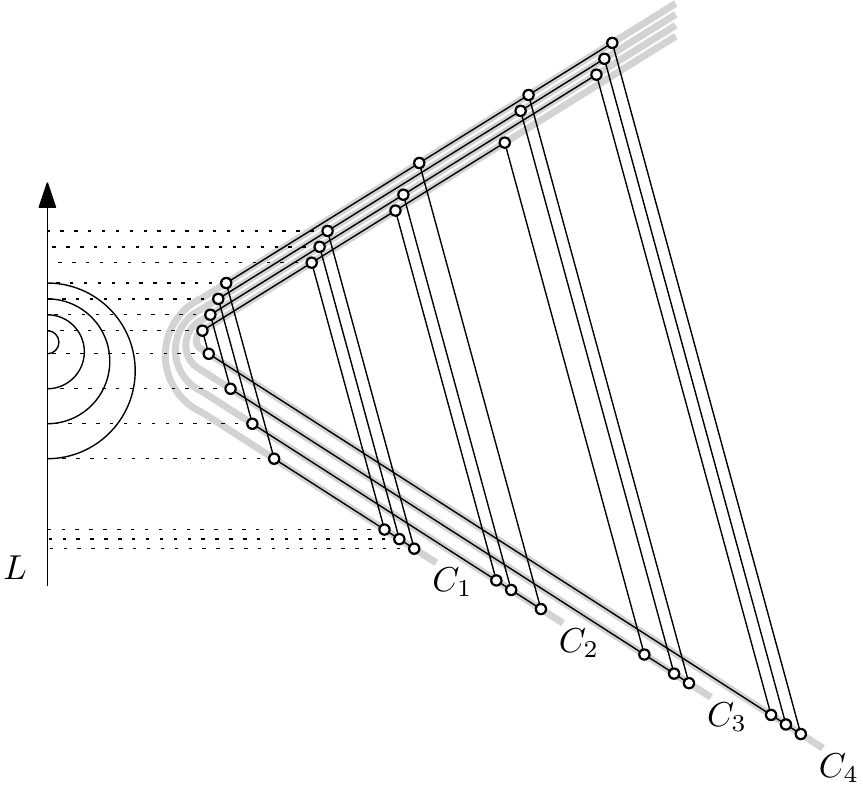}
 \caption{A poset $P$ of width $w=4$ (A partition into $4$ chains is indicated in grey.) and a linear extension
   $L$ (ordering the elements by their $y$-coordinates) of~$P$ with a $w^2$-rainbow.}
 \label{fig:bad-linear-extension}
\end{figure}

In fact we have shown a much stronger statement:
If $P$ and a chain partition $C_1,\ldots,C_w$ are given, then there is a
partition of the edges of the cover graph of $P$ into parts $Q_{i,j}$ with
$i,j\in [w]$ such that each $Q_{i,j}$ is a queue for every(!) linear extension
$L$ of $P$. 
Let us remark that for some posets $P$ and some linear extensions $L$ of $P$, the resulting queue layout indeed has a $\width(P)^2$-rainbow.
An example is indicated in~\cref{fig:bad-linear-extension}.

\subsection{Concepts needed for the construction}

Let $P$ be a poset. The \emph{dual} of $P$, denoted $\bar{P}$, is the poset on the same ground set such that: $x < y \text{ in } P \iff y < x \text{ in } \bar{P}$.
In terms of its diagram, the dual of~$P$ is obtained by flipping along a horizontal line.

A poset $P$ is \emph{2-dimensional} if and only if there are two linear extensions $L_1$ and $L_2$
such that: $x < y \text{ in } P \iff x < y \text{ in } L_1 \text{ and } L_2$.
Such a pair $L_1,L_2$ is called a \emph{realizer} of $P$.

When drawing 2-dimensional posets, it is common to represent each element~$x$ by a point with coordinates $(x_1,x_2)$
where $x_1$ is the position of $x$ in $L_1$ and $x_2$ is the position of $x$ in $L_2$, see~\cref{fig:dual+dim-B}.
This is also called a dominance drawing.

\calc_figscale{18}%
\begin{figure}[htb]
    \centerline{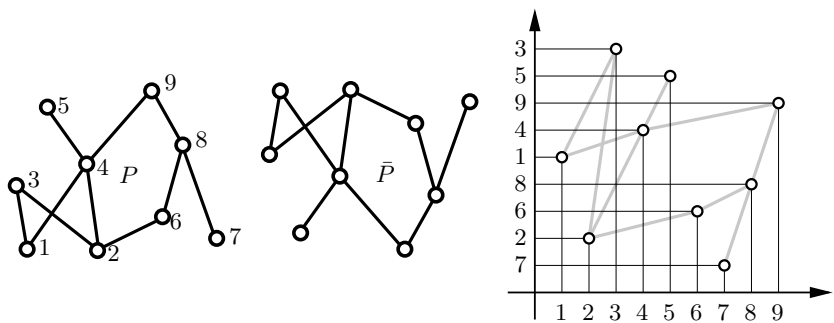}
    \caption{A poset $P$, its dual $\bar{P}$, and a 2-dimensional drawing of $P$.\label{fig:dual+dim-B}}
    \end{figure}%


\section{Proof of Theorem~\ref{thm:main}}

We define $P_w$ recursively, focusing on the recursive step.
As mentioned in the introduction, the recursive step involves lifting a given poset $P_{w-2}$ of width $w-2$ to the desired poset $P_w$ of width $w$ such that $\qn(P_w) \geq \qn(P_{w-2})+\lceil (w-1)/2 \rceil$.
Our lifting can be seen as an extension of the situation on the very right of \cref{fig:Alam-construction}.
Specifically, for $w\geq 3$, the construction of $P_w$ is based
on
\begin{itemize}
 \item a copy of $P_{w-2}$,
 \item a reinforcement poset $R_{w-2}$ of width $w-2$,
 \item two linear extensions $L_x$ and $L_y$ of $R_{w-2}$, and 
 \item the duals $\overline{P_{w-2}},\overline{R_{w-2}},\overline{L_x},\overline{L_y}$ of the above.
\end{itemize}
We invite the reader to take a look
at~\cref{fig:qn_bestPw}, which shows the construction of $P_w$ using $P_{w-2}$ and $R_{w-2}$ as a black box.
Formally, let $r=r(w-2)$ denote the number of elements in $R_{w-2}$.
Then, $P_w$ contains besides $P_{w-2},R_{w-2},\overline{P_{w-2}},\overline{R_{w-2}}$, two additional elements $a$ and $b$, and four chains of additional $r$ elements $x_1 < \cdots < x_r$, $y_1 < \cdots < y_r$, $\overline{x_r} < \cdots < \overline{x_1}$, and $\overline{y_r} < \cdots < \overline{y_1}$, together with the following additional relations:

\begin{itemize}
 \item $b$ is below $x_1,y_1$ and above $\overline{x_1},\overline{y_1}$.
 \item $a$ is above all elements in $P_{w-2}$ and below all elements in $\overline{P_{w-2}}$.
 \item All elements of $P_{w-2}$ are above all elements of $R_{w-2}$, and all elements of $\overline{P_{w-2}}$ are below all elements of $\overline{R_{w-2}}$.
 \item $x_i$ is above the $i$-th element in the linear extension $L_x$ of $R_{w-2}$, $i=1,\ldots,r$.
 \item $\overline{x_i}$ is below the $i$-th element in the dual $\overline{L_x}$ of $\overline{R_{w-2}}$.
 \item $y_i$ is above the $i$-th element in the linear extension $L_y$ of $R_{w-2}$, $i=1,\ldots,r$.
 \item $\overline{y_i}$ is below the $i$-th element in the dual $\overline{L_y}$ of $\overline{R_{w-2}}$.
 \item All relations that are transitively implied by the above.
\end{itemize}

First we observe that $\width(P_w) = \width(P_{w-2})+2 = w$, as $\width(P_{w-2})=\width(R_{w-2})=w-2$ and the additional elements (except $a$, which can be incorporated into an existing chain) can be covered by two chains.
Also note that the number $p(w)$ of elements of the poset $P_w$ is given by the recursion $p(w) = 2p(w-2) + 6r(w-2) +2$.
(Recall that $r(w-2)$ is the number of elements of $R_{w-2}$.)
Further note that $x_i$ and the $i$-th element of $L_x$ in $R_{w-2}$ indeed form a cover edge, as $L_x$ is a linear extension of $R_{w-2}$, $i=1,\ldots,r$.
Similarly for the edges between $R_{w-2}$ and $y_i$, as well as between $\overline{R_{w-2}}$ and $\overline{x_i},\overline{y_i}$, $i=1,\ldots,r$.


\begin{figure}[t]
 \centering
 \includegraphics{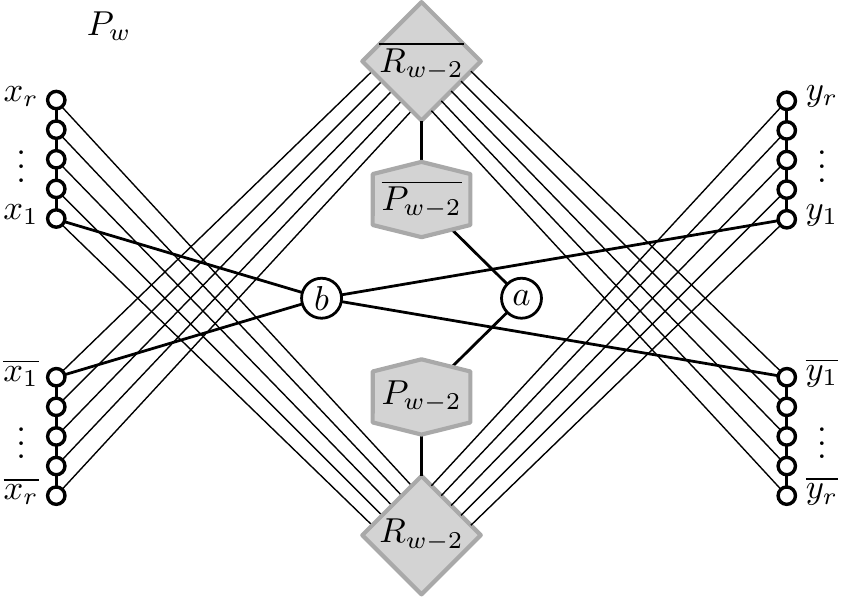}
 \caption{Recursive construction of $P_w$.}
 \label{fig:qn_bestPw}
\end{figure}

Furthermore, it can be seen that $P_w$ is self-dual; the reflection $P_w \leftrightarrow \overline{P_w}$ having two fixed points $a$ and $b$. This shows that when analyzing $\qn(P_w)$, we can restrict the attention to
linear extensions $L$ of $P_w$ which have $a$ before $b$.
With this assumption, a rainbow between $R_{w-2}$ and either $X = \{x_1,\ldots,x_r\}$ or $Y = \{y_1,\ldots,y_r\}$ nests above every rainbow of $P_{w-2}$.
See \cref{fig:a-before-b} for an illustration.
If we let $q_{w-2}$ be the
size of a rainbow between $R_{w-2}$ and either $X$ or $Y$, then we have the recursion:

\begin{equation}\label{eqn:r-sum}
  \qn(P_w) \geq  \qn(P_{w-2}) + q_{w-2}
\end{equation}

We think of this use of a self-dual construction as the \emph{symmetry trick}.
Again, let us mention that constructions given in~\cite{KMU-18} (proof of Prop.~2) and \cite{ABGKP-20} (proof of Thm. 4) also use a recursion based on two copies of the poset from the previous level of the recursion, as illustrated in the middle of \cref{fig:Alam-construction}. 
However, this only forces one edge to nest over the rainbow from the previous level of the recursion.
Our lifting forces a rainbow of edges whose size is linear in the width to nest over the previous level construction and its rainbow; thus giving overall a quadratic lower bound.

\begin{figure}[t]
 \centering
 \includegraphics{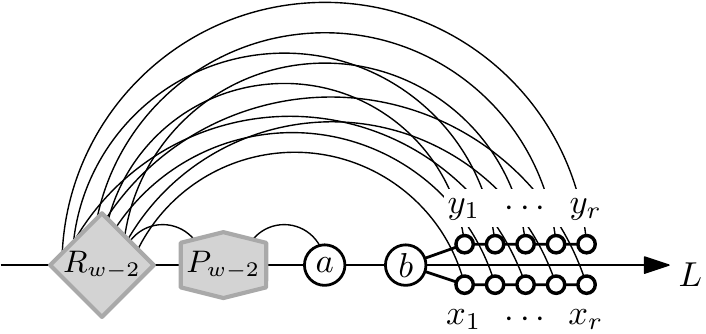}
 \caption{The general structure of a linear extension $L$ of $P_w$ with $a$ before $b$.}
 \label{fig:a-before-b}
\end{figure}

\bigskip

It remains to construct the poset $R_{w-2}$ together with two linear extensions~$L_x$ and $L_y$ such that in any linear extension $L$ having $R_{w-2}$ entirely before $X \cup Y$, a large rainbow between $R_{w-2}$ and either $X$ or $Y$ appears.
Recall that $q_{w-2}$ denotes the largest such rainbow and we seek to construct $R_{w-2}$ such that $q_{w-2}$ is at least linear in $w$.

As the elements in $X$ form a chain $x_1 < \cdots < x_r$ and thus are ordered in this way in $L$, rainbows between $R_{w-2}$ and $X$ are in bijection with subsets of elements in $R_{w-2}$ that are \emph{oppositely ordered} in $L$ and $L_x$.
Similarly, rainbows between $R_{w-2}$ and $Y$ appear when elements in $R_{w-2}$ are oppositely ordered in~$L$ and $L_y$.
Thus our goal is to construct $R_{w-2}$, $L_x$ and $L_y$ such that for every linear extension $L$ of $R_{w-2}$ there is a long increasing sequence in $L$ which is decreasing in $L_x$ or $L_y$.

To illustrate this idea, suppose that for each width $u < w$, we choose the poset $R_u$ to be an antichain of size $u$
and the linear extensions $L_x$ and $L_y$ to be a realizer (think of $L_x$ as the identity
permutation and of $L_y$ as its reverse). The Lemma of Erd\H os-Szekeres asserts that in every linear extension
of~$R_u$ there is an increasing or a decreasing sequence of size at least $\big\lceil\sqrt{u}\big\rceil$,
i.e., $q_u = \big\lceil\sqrt{u}\big\rceil$.

This value of $q_u$ together with Inequality~(\ref{eqn:r-sum}) yields

\[
 \qn(P_w) \geq \sum_{u < w;~ u\equiv w (2)}\Big\lceil\sqrt{u}\Big\rceil \quad\in \Theta(w^{3/2}).
\]

For the proof of the theorem we need a better construction for the reinforcement posets~$R_u$.
In particular, we seek to have $q_u \geq \lceil \frac{u+1}{2} \rceil$ instead of just $q_u \geq \big\lceil \sqrt{u} \big\rceil$.
A construction of such a $R_u$ is given in Subsection~\ref{subsec:Ru} and based on the following
lemma\footnote{The lemma with a different proof was discovered (but not yet published) in October 2020
    by the first and the second author together with Francois Dross, Piotr Micek, and Micha\l\ Pilipczuk.}. 

\begin{lemma}
  \label{lem:goodR}
For each $u\geq 1$, there is a 2-dimensional poset $R_u$ of width $u$ with a realizer $L_x,L_y$,  
such that if $L$ is a linear extension of $R_u$ and $d_x$ and $d_y$ denote the maximum lengths of an
increasing sequence in $L$ which is decreasing in $L_x$ and $L_y$ respectively, then $d_x + d_y \geq u+1$. 
\end{lemma}

The lemma says that we can assume the value $q_u = \lceil\frac{u+1}{2}\rceil$. With Inequality~(\ref{eqn:r-sum})
we get:

\[
\qn(P_w) \geq \sum_{u < w;~ u\equiv w (2)}\left\lceil\frac{u+1}{2}\right\rceil
\]

In the case $w$ odd, $w=2s+1$, we get $\qn(P_w) \geq \sum_{k=1}^s k = \binom{s+1}{2}$.
In the case~$w$ even, $w=2s$, we get $\qn(P_w) \geq \sum_{k=2}^s k = \binom{s+1}{2}-1$.
A simple computation shows that for $w\geq 4$ we get $\qn(P_w) \geq w^2/8$, independent of the parity of $w$.
This completes the proof of \cref{thm:main}.

\bigskip

The base of our recursive construction is the case $w=1$ or $w=2$, depending on the parity of $w$.
For the validity of \cref{thm:main}, it is enough to let $P_w$ with $w \in \{1,2\}$ be any poset of width $w$.
Of course, it is beneficial to start with a higher queue-number, also given that our bound of $w^2/8$ is less than the $w+1$ of Alam et al.~\cite{ABGKP-20} for small $w$.
The best results are achieved by starting at width~$3$ or $4$ (depending on the parity of the target width $w$) with the poset of Alam et al.~\cite{ABGKP-20} with queue-number $4$, respectively $5$.

\subsection{The construction of $R_u$ for Lemma~\ref{lem:goodR}}\label{subsec:Ru}

The construction of $R_u$ is again recursive.
Let $R_1$ be a single element.
Then clearly $d_x + d_y = 2$. 
For the construction of $R_u$ for $u\geq 2$ we again use the symmetry trick.
We take two copies $Q_1,Q_2$ of $R_u$ and two additional elements~$a$ and $b$.
Then $R_u$ is obtained by a series composition of $Q_1 + a + Q_2$, and a parallel composition of the result with element $b$.
Formally,

\begin{itemize}
 \item $a$ is above every element of $Q_1$ and below every element of $Q_2$, while
 \item $b$ is incomparable to all other elements.
\end{itemize}

The two linear extensions of the realizer $L_x,L_y$ of $R_u$ are obtained as follows.

\begin{itemize}
 \item $L_x = b, L_x(Q_1), a, L_x(Q_2)$
 \item $L_y = L_y(Q_1), a, L_y(Q_2), b$,
\end{itemize}

where $L_x(Q_i),L_y(Q_i)$ is the realizer of the copy $Q_i$ of $R_{u-1}$, $i=1,2$.
We invite the reader to look at \cref{fig:R-construction} for two illustrations of this recursive construction step for $R_u$ and its realizer $L_x,L_y$.

First, we observe that $\width(R_u) = \width(R_{u-1})+1 = u$, as element $b$ can be covered by a new chain and element $a$ can be incorporated into an existing chain.
Also note that the number $r(u)$ of elements in $R_u$ is given by the recursion $r(u) = 2r(u-1)+2$, which with $r(1)=1$ solves for $r(u) = \frac32\cdot 2^u - 2$.
Further observe that $R_u$ is again self-dual.
In particular the two copies $Q_1$ and $Q_2$ of~$R_{u-1}$ are isomorphic.
The reflection $R_u \leftrightarrow \overline{R_u}$ has two fixed points $a$ and $b$.

\begin{figure}[t]
 \centering
 \includegraphics[scale=0.9]{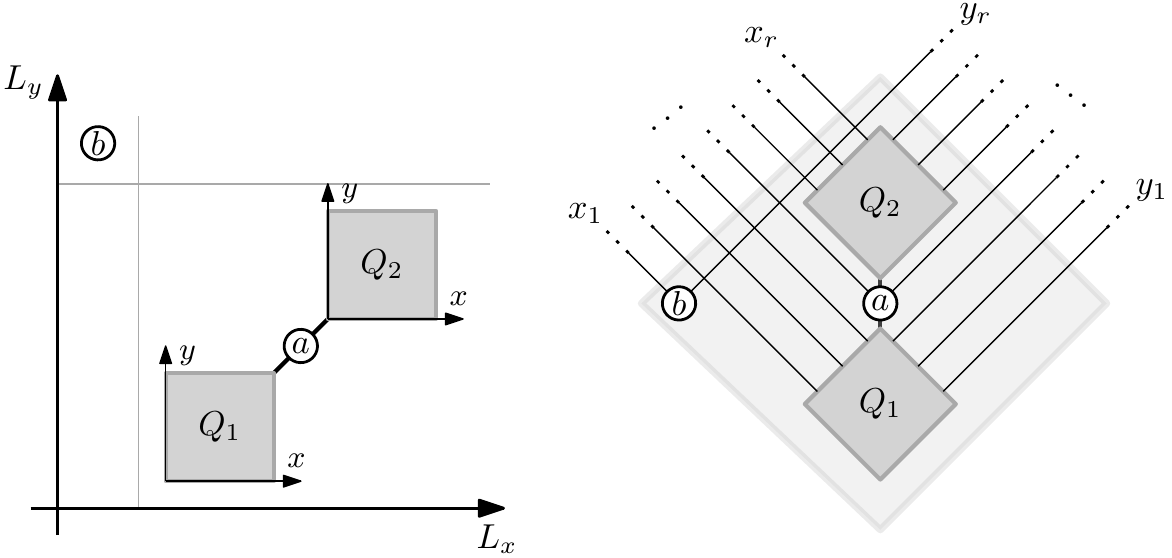}
 \caption{The recursive construction of $R_u$ with its realizer $L_x,L_y$.}
 \label{fig:R-construction}
\end{figure}

Now let $L$ be any linear extension of $R_u$. First suppose that $a < b$ in $L$.
Let~$L'$ be the restriction of $L$ to $Q_1$. By induction
the lengths $d'_x$ and $d'_y$ of increasing sequences of $L'$
which are decreasing in the two linear extensions of the realizer $L_x(Q_1),L_y(Q_1)$ of $Q_1$
satisfy $d'_x+d'_y\geq u$. Since $b$ precedes $Q_1$ in $L_x$ and comes after
$Q_1$ in $L$, we have $d_x \geq d'_x +1$. Together with the trivial $d_y \geq d'_y$,
we get $d_x+d_y \geq u+1$.

If we have $b<a$ in $L$, then we consider $Q_2$. As before we get the two values $d'_x$ and $d'_y$
for the restriction $L'$ of $L$ to $Q_2$ and know by induction that  $d'_x+d'_y\geq u$.
This time $b$ precedes $Q_2$ in $L$ but comes after $Q_2$ in $L_y$, which gives $d_y \geq d'_y + 1$.
Together with the trivial $d_x \geq d'_x$ we again see that $d_x+d_y \geq u+1$.
This completes the proof of \cref{lem:goodR}.

\bigskip

We remark that in both the construction of $P_w$ and $R_u$, the element $a$ is used only for the sake of the exposition.
It would suffice to add all relations between~$P_{w-2}$ and $\overline{P_{w-2}}$, respectively $Q_1$ and $Q_2$.
While this gives slightly smaller constructions for $P_w$ and $R_u$, they would still be exponential in their width.
(Recall that $r(u)=\frac32\cdot 2^u-2$ and hence $p(w) = 2p(w-2)+6r(w-2)+2 = \Theta(2^w)$.)

\section{Conclusions}

We have made substantial progress in the understanding of queue-numbers of partially ordered sets.
We take the opportunity to list and comment on open questions in the field.
\begin{itemize}
\item An obvious question is to ask for improved upper and lower bounds.  More
  precisely, we now know that the growth rate of the maximum queue-number of posets of
  width $w$ is $(C + o(1))w^2$ for some constant $C$ between $1/8$ and~1.
  What is the precise value of constant $C$?
\item Our reinforcement poset $R_u$ is 2-dimensional for every $u$.
 However our entire lower bound example $P_w$ is not (already for $w=3$), and the same holds for the example of Alam et al. in the left of \cref{fig:Alam-construction}.
 We think it is interesting to see whether there exists any 2-dimensional poset $P$ with $\qn(P) \geq \width(P)+1$.
\item What is the maximum queue-number of posets of width $w$ with a planar diagram?
  Knauer, Micek, and the second author~\cite{KMU-18} proved the lower bound~$w$ by observing that the simple lifting operation in the middle of \cref{fig:Alam-construction} preserves planarity, while their upper bound is $3w-2$.
  Clearly, the better lifting operation introduced here necessarily introduces crossing cover edges.
\item Heath and Pemmaraju~\cite{HP-97} conjectured that planar posets on
  $n$ elements have queue-number at most $\sqrt{n}$. Their lower bound construction is an $r$-antichain~$R$ with realizer $L_x,L_y$ together
  with an $r$-chain $X = x_1 < \cdots < x_r$ matched upward in order of $L_x$ and an $r$-chain $Y = y_1 < \cdots < y_r$ matched downward in order of $L_y$; see \cref{fig:planar-lower-bound}.
  The Lemma of Erd\H os-Szekeres implies for this planar poset $P$ with $n = 3r$ elements that $\qn(P) \geq \big\lceil\scalebox{0.85}{$\sqrt{n/3}$}\big\rceil$.
  It is open whether there is an asymptotically matching upper bound.
\item Dujmovi\'c and Wood~\cite{DW-04} show that a random vertex ordering for an undirected graph $G$ has with positive
  probability no rainbow of size $\big\lceil e\sqrt{m} \big\rceil$, where~$e$ is the base of the natural logarithm and $m$ is the number of edges in~$G$.
 Can a similar result be obtained by considering a random linear extension of a poset $P$?
 Note that a positive answer would resolve (up to a constant factor) the previous question of Heath and Pemmaraju about planar posets.
\item
  In~\cite{KMU-18} is was shown that posets $P$ of width 2 have $\qn(P) \leq 2$.
  In~\cite{ABGKP-20} it was shown that posets $P$ of width 3 may have $\qn(P) \geq 4$
  and satisfy $\qn(P) \leq 5$. Is 4 or 5 the best upper bound in this case?
\end{itemize}
  
\begin{figure}[ht]
 \centering
 \includegraphics{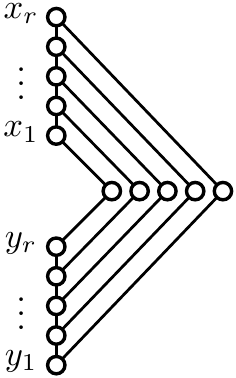}
 \caption{Heath and Pemmaraju's construction~\cite{HP-97} of a planar poset $P$ on $n=3r$ elements with
   $\qn(P) \geq \big\lceil\scalebox{0.85}{$\sqrt{n/3}$}\big\rceil$.}
 \label{fig:planar-lower-bound}
\end{figure}

\bibliographystyle{splncs04}
\bibliography{lit-Torsten}

\end{document}